\documentclass{amsart}
\numberwithin{equation}{section}
\usepackage{amssymb}
\usepackage{amsmath}
\usepackage{amsthm}
\usepackage{amscd}
\usepackage{amsfonts}
\usepackage{ascmac}
\usepackage[usenames]{color}
\usepackage{graphics}

\theoremstyle{plain}
\newtheorem{theorem}{Theorem}[section]
\newtheorem{proposition}[theorem]{Proposition}

\newtheorem{lemma}[theorem]{Lemma}

\newtheorem{definition}[theorem]{Definition}

\newtheorem{remark}[theorem]{Remark}

\newcommand{\bfC}{{\mathbf C}}

\newcommand{\bfR}{{\mathbf R}}
\newcommand{\bfZ}{{\mathbf Z}}

\newcommand{\bari}{{\overline i}}
\newcommand{\barj}{{\overline j}}

\newcommand{\barl}{{\overline \ell}}

\newcommand{\barpsi}{{\overline \psi}}
\newcommand{\barvarphi}{{\overline{\varphi}}}
\newcommand{\barpartial}{{\overline \partial}}

\newcommand{\wtv}{{\widetilde v}}

\newcommand{\mapright}[1]{\smash{\mathop{   \hbox to 0.7cm{\rightarrowfill}}
  \limits^{#1}}}

\newcommand{\Ker}{{\rm Ker}}

\newcommand{\Ric}{\operatorname{Ric}}

\newcommand{\grad}{\mathrm{grad}}

\newcommand{\Fut}{\mathrm{Fut}}

\newcommand{\lb}{\left (}
\newcommand{\rb}{\right )}

\newcommand{\Aut}{\operatorname{Aut}}
\newcommand{\Isom}{\operatorname{Isom}}
\newcommand{\TildeT}{\widetilde{T}}
\newcommand{\Tildefrakt}{\widetilde{\frak t}}
\title{Deformations of Fano manifolds with weighted solitons}
\author{Akito Futaki}
\address{Yau Mathematical Sciences Center, Tsinghua University, Haidian district, Beijing 100084, China}
\email{futaki@tsinghua.edu.cn}

\date{September 7, 2023}

\begin{document}

\begin{abstract}
We consider weighted solitons on Fano manifolds which include K\"ahler-Ricci solitons, Mabuchi solitons and 
base metrics inducing Calabi-Yau cone metrics outside the zero sections of the canonical line bundles
(Sasaki-Einstein metrics on the associated $U(1)$-bundles). 
In this paper, we give a condition for a weighted soliton on a Fano manifold $M_0$ 
to extend to weighted solitons on small deformations $M_t$ of the Fano manifold $M_0$. 
More precisely,
we show that all the members $M_t$ of the Kuranishi family of a Fano manifold $M_0$ with a weighted soliton have weighted solitons if and only if the dimensions of $T$-equivariant automorphism groups of $M_t$ 
are equal to that of $M_0$, and also 
if and only if the $T$-equivariant automorphism groups of $M_t$ are all isomorphic to that of $M_0$, 
where the weight functions are defined on the moment polytope of the Hamiltonian $T$-action. This 
generalizes a result of Cao-Sun-Yau-Zhang
for K\"ahler-Einstein metrics.
\end{abstract}

\maketitle


\section{Introduction}


Let $M$ be a Fano manifold, i.e. a compact complex manifold with positive first Chern class, of complex dimension $m$.
We regard $2\pi c_1(M)$ as a K\"ahler class.
The K\"ahler form $\omega$ is expressed as
$$ \omega = \sqrt{-1}\, g_{i\barj}\, dz^i \wedge dz^\barj$$
and the K\"ahler metric $g_{i\barj}$ is often identified with the K\"ahler form $\omega$.
Let $T$ be a real compact torus in the automorphism group $\Aut (M)$, and assume that $\omega$ is $T$-invariant. Since $M$ is Fano and simply connected
the $T$-action is Hamiltonian with respect to $\omega$. Since the $T$-action naturally lifts to the anti-canonical line bundle
$K_M^{-1}$ we have a canonically normalized moment map $\mu_\omega : M \to \mathfrak t^\ast$
where $\mathfrak t$ is the Lie algebra of $T$ and $\mathfrak t^\ast$ its dual space, c.f. Appendix in \cite{FutakiZhang18}. Let
$\Delta := \mu_{\omega}(M)$ be the moment polytope. Then $\Delta$ is independent of $\omega \in 2\pi c_1(M)$.
Let $v$ be a positive smooth function on $\Delta$. Regarding $\mu$ as coordinates on $\Delta$ using the action angle
coordinates, we may sometimes write $v(\mu)$ instead of $v$. The pull-back $\mu^\ast_\omega v$ is a smooth function on $M$, and
for this we write $v(\mu_\omega)=\mu^\ast_\omega v = v \circ \mu_\omega$. 

We say that a K\"ahler metric $\omega$ in $2\pi c_1(M)$ a {\it weighted $v$-soliton} or simply {\it $v$-soliton} if
$$ \Ric(\omega) - \omega = \sqrt{-1} \partial\barpartial \log v(\mu_\omega)$$ 
where $\Ric(\omega) = -i \partial\barpartial \log \omega^m$ is the Ricci form.
We also call $\omega$ simply a {\it weighted soliton}
when it is $v$-soliton for some $v$, or when $v$ is obvious from the context.
Examples of weighted solitons are a K\"ahler-Ricci soliton when $v(\mu) = e^{\langle\mu,\xi\rangle}$ for some $\xi \in \mathfrak t$, 
a Mabuchi solitons when $v(\mu) = \langle\mu, \xi\rangle + a$ for some positive constant $a$, and a basic metric which 
which induces Calabi-Yau cone metrics outside the zero sections of the canonical line bundle and hence
a Sasaki-Einstein metrics on the $U(1)$-bundle of $K_M^{-1}$ when $v(\mu) = \lb \langle\mu,\xi\rangle + a\rb^{-m-2}$, see \cite{Inoue19}, \cite{Lahdili18}, \cite{ACL21}, \cite{AJL21}, \cite{HanLi20}, 
\cite{LiChi21}.

In this paper we consider the Kuranishi family 
$\varpi : \mathfrak M \to B$ of deformations of a Fano manifold $M$
which is a complex analytic family of Fano manifolds where $B$ is an open set in $\bfC^n$
containing the origin $0$ and we write $M_t:= \varpi^{-1}(t)$ and require $M_0 = M$, c.f. \cite{Kodaira86}, \cite{KodairaSpencer58}, \cite{Kuranishi64}, \cite{MorrowKodaira}, 
\cite{Sun12}, \cite{FSZ22}.
Note that there is no obstruction
for Fano manifolds since
$$H^2(M_0,\Theta) \cong H^{m-2}(M_0, \Omega^1(K_{M_0})) =  0$$
by Serre duality and Kodaira-Nakano vanishing. 
For a given K\"ahler form $\omega \in 2\pi c_1(M)$
let $f \in C^\infty (M)$ satisfy
\begin{equation*}\label{Kura1}
\Ric(\omega) - \omega = \sqrt{-1} \partial\barpartial f.
\end{equation*}
The Kuranishi family we consider in this paper is described by a family of vector valued $1$-forms 
parametrized by $t \in B$
$$\varphi(t) = \sum_{i=1}^k t^i\varphi_i + \sum_{|I|\ge2}t^I\varphi_I\ \in\ A^{0,1}(T^\prime M)$$
such that 
\begin{equation}\label{Kura}
\begin{cases}
 \barpartial\varphi(t) = \frac12 [\varphi(t),\varphi(t)];\\
 \barpartial^\ast_f \varphi(t) = 0;\\
 \varphi_1, \cdots, \varphi_k\ \text{form a basis of the space of all}\ T^\prime M\text{-valued}\ \Delta_f\text{-harmonic}\ (0,1)\text{-forms}\\
\end{cases}
\end{equation}
where $\Delta_f = \barpartial_f^\ast\barpartial + \barpartial\,\barpartial_f^\ast$ is the weighted Hodge Laplacian with
$\barpartial_f^\ast$ the formal adjoint of $\barpartial$ with respect to the weighted $L^2$-inner product
$\int_M (\cdot,\cdot) e^f \omega^m$.
See \cite{FSZ22} for more detail about this Kuranishi family. We showed in \cite{FSZ22} that the K\"ahler form 
$\omega$ on $M_0 = M$ remains to be a K\"ahler form
on $M_t$. The main result of this paper is stated as follows.
\begin{theorem}\label{Main Thm} Suppose that $M_0$ has a weighted $v$-soliton. 
Consider the Kuranishi family \eqref{Kura} with $f=\log v(\mu_\omega)$. Then, shrinking $B$ if necessary,  the following statements are equivalent.
\begin{enumerate}
\item[(1)]  $M_t$ has a weighted $v$-soliton for all $t\in B$.
\item[(2)] $T$ is included in $\Aut(M_t)$, and for the centralizer $\Aut^T(M_t)$ of $T$ in $\Aut(M_t)$, $\dim \Aut^T(M_t) = \dim \Aut^T(M_0)$ for all 
$t \in B$.
\item[(3)] $T$ is included in $\Aut(M_t)$, and the identity component $\Aut_0^T(M_t)$ of $\Aut^T(M_t)$ is isomorphic to $\Aut_0^T(M_0)$ for all 
$t \in B$.
\end{enumerate}
\end{theorem}

Although there are extensive studies on the existence of K\"ahler-Einstein metrics on Fano manifolds,
e.g. \cite{ACCFKMSSV}, 
currently, there are not many existence results on weighted solitons on Fano manifolds. 
Because of this lack of examples of weighted solitons, it is not easy to find non-trivial applications
of Theorem \ref{Main Thm}. 

As for the deformations of complex structures of polarized manifolds with weighted csck metrics, there is a result
by Hallam \cite{Hallam22} which states that a small deformation of a polarized manifold with a weighted csck metric
has a weighted cscK metric if and only if it is weighted K-polystable
with respect to smooth $T$-equivariant test configurations. The ``only if'' part follows from
a result of Apostolov-Jubert-Lahdili \cite{AJL21}. The result of Hallam extends the results of
Br\"onnle \cite{Bro11} and Szekelyhidi \cite{Sze10} for cscK metrics.

The outline of the proof of Theorem \ref{Main Thm} is as follows. The proof of Theorem \ref{Main Thm} above is
largely parallel to that of Theorem 1.1 of Cao-Sun-Yau-Zhang \cite{CaoSunYauZhang2022}. Just as the notion of K\"ahler-Einstein metrics are
generalized to constant scalar curvature K\"ahler (cscK for short) metrics and further to extremal K\"ahler metrics (\cite{calabi85}), the notion of 
weighted solitons are generalized to weighted cscK metrics and further to weighted extremal metrics (\cite{Lahdili18}, \cite{Inoue19}, \cite{Inoue22}). In Section 2, we extend a result of Rollin-Simanca-Tipler \cite{RST} for extremal K\"ahler 
metrics to
show that a weighted extremal metric on $M_0$ can be extended to weighted extremal metrics 
on small deformations $M_t$ if the maximal torus in the reduced automorphism group acts
on $B$ trivially. In Section 3 we review how weighted solitons are regarded as weighted cscK metrics.
In Section 4 we finish the proof of Theorem \ref{Main Thm}. We first review results in \cite{FSZ22} about the K\"ahler forms and their Ricci potentials for the
Kuranishi family. Next, we show a lemma which implies that the action of the maximal torus on $B$ is trivial
so that we can apply the result obtained in Section 2. We then show using the formula of the Ricci potential
obtained in \cite{FSZ22} that the weighted extremal metrics obtained in Section 2 are in fact weighted cscK metrics
which are in this case weighted solitons, proving (2) implies (1). That (1) implies (3) is proved using the
K-polystability characterization obtained 
by the works of Han-Li \cite{HanLi20}, Li \cite{LiChi21}, Blum-Liu-Xu-Zhuang \cite{BLXZ}
and closely following the arguments of \cite{CaoSunYauZhang2022}. That (3) implies (2) is trivial.

\section{Weighted scalar curvature}
In this section we review the weighted scalar curvature, which is also called the $(v,w)$-scalar curvature, introduced by Lahdili \cite{Lahdili18}, see also Inoue \cite{Inoue19}, \cite{Inoue22} for a similar idea.

Let $M$ be a compact K\"ahler manifold and $\Omega$ its K\"ahler class. Recall that the Lie algebra of $\Aut(M)$ is the Lie algebra $\mathfrak h(M)$ of all holomorphic
vector fields. We denote by $\Aut_r(M) \subset \Aut(M)$ the reduced automorphism group, i.e. the Lie algebra $\mathfrak h_r(M)$ of $\Aut_r(M)$
consists of holomorphic vector fields with non-empty zeros. They are obtained in the form $\grad'u$, i.e. the $(1,0)$-part of the gradient vector field, of some complex valued smooth functions $u$, see e.g. \cite{lebrunsimanca94}.

Let $T$ be a compact real torus in $\Aut_r(M)$. 
As in the Introduction $\Aut_r^T(M)$ denotes the centralizer of $T$ in $\Aut_r(M)$, i.e. the subgroup consisting of
$T$-equivariant automorphisms. In the Fano case $\Aut_r^T(M) = \Aut^T(M)$.
Let $\omega \in \Omega$ be a $T$-invariant K\"ahler form. Then $T$ acts on $(M,\omega)$ in the Hamiltonian way. 
Let $\mu_\omega : M \to \mathfrak t^\ast$ be the moment map where $\mathfrak t$ is the Lie algebra of $T$ and $\mathfrak t^\ast$ its dual space. Then $\Delta:=\mu_\omega(M)$ is a compact convex polytope.
This is independent of $\omega \in \Omega$ up to translation, but the ambiguity of translation is fixed by giving a normalization of $\mu_\omega$
which specifies the average by the integration.
Let $v$ be a positive smooth function on $\Delta$. 

As in Section 1, we also write $v = v(\mu)$ as a function on
$\Delta$ by considering $\mu$ to constitute the action-angle coordinates, 
and also write $v(\mu_\omega) := \mu_\omega^\ast v$ as a positive smooth function on $M$.
We define {\it $v$-scalar curvature} $S_v(\omega)$ of a $T$-invariant K\"ahler form $\omega$ by
$$ S_v(\omega):= v(\mu_\omega) S(\omega) + 2 \Delta_\omega v(\mu_\omega) 
+ \langle g_\omega,\mu_\omega^\ast Hess(v)\rangle$$
where $S(\omega)$ denotes the K\"ahler geometers' scalar curvature 
$$ S(\omega) = - g^{i\barj} \frac{\partial^2 }{\partial z^i \partial z^\barj} \log \det (g_{l\barl})$$
of $\omega$, $\Delta_\omega = \barpartial^\ast \barpartial$ the Hodge $\barpartial$-Laplacian on functions, and 
$$\langle g_\omega,\mu_\omega^\ast Hess(v)\rangle = g^{i\barj} v_{\alpha\beta}\mu^\alpha_i\mu^\beta_\barj$$
is the trace of the pull-back by $\mu_\omega$ of
 the Hessian $Hess(v)$ of $v$ on $\mathfrak t^\ast$ in which we express the moment map $\mu_\omega : M \to \mathfrak t^\ast$ as
$\mu_\omega(p) = (\mu^1(p), \cdots, \mu^\ell(p))$ with $d\mu^\alpha = i(X^\alpha)\omega$ for a basis $X^1, \cdots, X^\ell$ of $\mathfrak t$.
Thus, our $S_v$ is half of that in \cite{Lahdili18}.

Let $w$ be another positive smooth function on $\Delta$.
We define {\it $(v,w)$-scalar curvature} $S_{v,w}$ by
$$ S_{v,w} = \frac{S_v}{w(\mu_\omega)}.$$
The notion of $S_{v,w}$-scalar curvature was originally introduced as a generalization of
conformally K\"ahler, Einstein-Maxwell metrics after extensive studies such as \cite{LeBrun16}, \cite{AM},
\cite{FO17}, \cite{FO_reductive17}, 
\cite{Lahdili17_1}, \cite{Lahdili17_2}. Later it turned out that 
the $(v,w)$-cscK metrics include much more unexpected examples as mentioned in Section 1,(Hence if $g$ is a $(v,w)$-extremal metric then 
$\grad' S_{v,w} \subset \widetilde{\mathfrak t}$. But we do not assume this for the moment.
see also Section 3.

We call $g$ a {\it weighted extremal metric} or {\it $(v,w)$-extremal metric} if 
$$\grad'S_{v,w} = g^{i\barj}\frac{\partial S_{v,w}}{\partial z^\barj}\frac{\partial}{\partial z^i}$$
is a holomorphic vector field.

\begin{remark} In Section 3.2 of \cite{Lahdili18}, Lahdili also defined $(v,w)$-extremal K\"ahler
metrics. But his definition is slightly different from ours in that the extremal vector field belongs to $\mathfrak t$
in his case
but does not in our case.
\end{remark}

Define $L_v\varphi$ for complex valued smooth functions $\varphi$ by
$$ L_v \varphi = \nabla^i\nabla^j(v(\mu_\omega)\nabla_i\nabla_j \varphi),$$
and call $L_v$ the $v$-twisted Lichnerowicz operator. Obviously, $L_v$ is self-adjoint elliptic operator;
$$ \int_M (L_v\varphi)\, \barpsi\, \omega^m =  \int_M \varphi\, \overline{L_v \psi}\, \omega^m$$
where $m = \dim M$.
$L:=L_1$ is the standard Lichnerowicz operator. 
The kernel of $L_v$ consists of complex valued smooth functions $u$ such that $\grad'u$ is a holomorphic vector fields,
an thus $\Ker\, L_v  = \mathfrak h_r(M) = Lie(\Aut_r(M))$. We also define $L_{v,w}$ by
$$ L_{v,w} = \frac 1{w(\mu_\omega)}\, L_v.$$

Consider the one parameter family of metrics $g_{ti\barj} = g_{i\barj} + t\varphi_{i\barj}$. 
By straightforward computations one can show
\begin{equation}\label{Lich1}
\left.\frac d{dt}\right|_{t=0} S_v(g_t) = - L_v \varphi + S_v^i\,\varphi_i,
\end{equation}
\begin{equation}\label{Lich2}
\left.\frac d{dt}\right|_{t=0} S_{v,w}(g_t) = - L_{v,w} \varphi + S_{v,w}^i\,\varphi_i.
\end{equation}
It is also straightforward using \eqref{Lich2} to show the following Proposition \ref{Lich3}.
\begin{proposition}\label{Lich3}
A critical point of the weighted Calabi functional 
$$g \mapsto \int_M S_{v,w}^2 (g)\,w(\mu_{\omega})\, \omega^m$$ 
is a weighted extremal metric.
\end{proposition}
As in \cite{futaki83.1} we can define the following invariants.
\begin{proposition}\label{Lich3.1}
Let $h_X \in \Ker L_v$ be the real Killing potential of $X \in \mathfrak t$, 
i.e. $i \grad'h_X = X'$.
Then $\Fut_v$ and $\Fut_{v,w}$ defined by
\begin{equation}\label{Lich4}
 \Fut_v (X) = \int_M (S_v - c_v)\, h_X\,\omega^m,
 \end{equation}
 and 
 \begin{equation}\label{Lich5}
 \Fut_{v,w} (X) = \int_M (S_{v,w} - c_{v,w})\, h_X\,w(\mu_\omega)\,\omega^m
 \end{equation}
are independent of choice of $\omega \in \Omega$ where 
$c_{v,w} = \int_M S_v\, \omega^m/\int_M w(\mu_\omega)\, \omega^m$ 
and $c_{v} = c_{v,1}$
which
are independent of $\omega \in \Omega$. 
\end{proposition}
\begin{proof} If $h_X(\omega)$ is the real Killing potential as in the statement of the proposition for $\omega \in \Omega$ then,
under the normalization $\int_M h_X\,\omega^m = 0$, we have
$h_X(\omega_t) = h_X + th_X^i\varphi_i$ where $\omega_t$ is the K\"ahler form of $g_t$ which was defined two lines above
the equation \eqref{Lich1}.
Hence
\begin{eqnarray*}
\left.\frac d{dt}\right|_{t=0} \int_M S_v(\omega_t)\,h_X(\omega_t)\,\omega_t^m &=&
\int_M \lb \lb -L_v \varphi + S_v^i\varphi_i\rb h_X + S_v h_X^i \varphi_i + S_v h_X \Delta \varphi\rb \omega^m \\
&=& - \int_M \varphi L_v h_X \omega^m = 0,
\end{eqnarray*}
and 
\begin{eqnarray*}
\left.\frac d{dt}\right|_{t=0} \int_M h_X(\omega_t)\,\omega_t^m = 0.
\end{eqnarray*}
Thus $\Fut_v$ is independent of $\omega \in \Omega$. 
Note however that the expression of \eqref{Lich4} does not depend on the normalization of $h_X$.
By a similar computation one can show that
$\Fut_{v,w}$ is independent of $\omega \in \Omega$.
\end{proof}

\begin{remark}
In \cite{Lahdili18} Lahdili shows for smooth test configurations, 
the slope of the weighted
Mabuchi functional is the weighted Donaldson-Futaki invariant.
Applying
this to the case of product test configurations also yields
Proposition 2.2.
\end{remark}

\begin{remark}\label{Lich5.1}If $g$ is a $(v,w)$-extremal metric with non-constant $S_{v,w}$ then
$$ \Fut_{v,w}(J\grad S_{v,w}) = \int_M (S_{v,w} - c_{v,w})^2\, w(\mu_\omega)\,\omega^m > 0.$$
\end{remark}
\begin{remark}A decomposition theorem  similar to that proved by Calabi \cite{calabi85} for extremal K\"ahler metrics holds for weighted extremal K\"ahler metrics,
see Theorem B.1 in \cite{Lahdili18}, also \cite{FO_reductive17}, \cite{Lahdili17_1}. 
A consequence of this is that, if $g$ is a weighted extremal metric, then the centralizer of $\grad'S_{v,w}$ in $\mathfrak h_r^T(M)$ is the complexification of the real Lie algebra of all
$T$-equivariant Killing vector fields with non-empty zeros. In particular, if $g$ has constant $(v,w)$-scalar curvature, then 
the identity component of $\Aut_{r}^T(M)$ is the
complexification of the identity component of $\Isom_r^T(M)$ consisting of isometries with non-empty fixed point set. 
\end{remark}

Let $(M,g)$ be a compact K\"ahler manifold of complex dimension $m$.
Let $\TildeT$ be the maximal torus in $\Aut_r (M)$ including $T$. 
Let $L^2_k(M)$ be the $k$-th Sobolev space with respect to the metric $g$ with weight $w(\mu_\omega)$. 
Here the weight $w(\mu_\omega)$ means that $L^2$-inner product is given by 
$$ (\phi,\psi) = \int_M\ \phi\psi\,w(\mu_\omega)\, \omega^m.$$
We take $k$ sufficiently large
so that $L^2_k(M)$-functions form an algebra. Let $L^2_{k,\TildeT}(M)$ be the subspace of $L^2_k(M)$ consisting of $\TildeT$-invariant functions.

Let $\mathcal H_g$ be the space of Killing potentials corresponding to $\Tildefrakt$. Then the functions in $\mathcal H_g$ are purely imaginary.
As in Proposition \ref{Lich3.1} we call a function in $i\mathcal H_g$ a real Killing potential.
Let $W_{k,g}$ be the orthogonal complement of $i\mathcal H_g$ in $L^2_{k,\TildeT}$:
$$ L^2_{k,\TildeT} = i\mathcal H_g \oplus W_{k,g}$$
with $L^2$-orthogonal projections
$$\pi_g^H : L^2_{k,\TildeT} \to i\mathcal H_g, \qquad \pi_g^W : L^2_{k,\TildeT} \to W_{k,g}.$$
Then $\pi^H_g(S_{v,w}(g))$ is a smooth function independent of the choice of $k$, and the gradient vector field $\grad' \pi^H_g(S_{v,w}(g))$
is independent of the choice of $g$ in the same K\"ahler class, see Theorem 3.3.3 in \cite{futaki88}, also \cite{futakimabuchi95}, and also Theorem 1.5 in 
\cite{Yaxiong22} for conformally K\"ahler Einstein-Maxwell metrics. The proof in the weighted $(v,w)$-case is identical to \cite{Yaxiong22}.
This vector field is called the {\it extremal K\"ahler vector field}. If $g$ is a weighted extremal metric then 
$$ \pi^H_g(S_{v,w}(g)) = S_{v,w}(g).$$
Hence $g$ is a weighted extremal metric if and only if
$$ \pi^W_g(S_{v,w}(g)) = 0.$$
\begin{definition}\label{reducedS} We call $S^{\mathrm{\mathrm{red}}}_{v,w}(g):=\pi^W_g(S_{v,w}(g))$ the reduced $(v,w)$-scalar curvature, or simply
reduced scalar curvature of $g$. Thus, $g$ is a weighted $(v,w)$-extremal metric if and only if $S^{\mathrm{red}}_{v,w}(g)=0$.
\end{definition}
We can then modify \eqref{Lich2} as
\begin{equation}\label{Lich2.1}
\left.\frac d{dt}\right|_{t=0} S^{\mathrm{red}}_{v,w}(g_t) = - L_{v,w} \varphi + S_{v,w}^{\mathrm{red}\,i}\,\varphi_i.
\end{equation}


Let $\varpi : \mathcal M \to B$ a complex analytic family of complex deformations with $M_0 = M$ where $B$ is an open set in $\bfC^k$ containing $0$ and we put $M_t :=\varpi^{-1}(t)$. 
We assume $(M_0,g_0)=(M,g)$ is a compact $(v,w)$-extremal K\"ahler manifold.
By the rigidity theorem of Kodaira-Spencer, $M_t$ is K\"ahler for all small $t$, see e.g. \cite{MorrowKodaira}. 
Note that the dimension of the Dolbeault 
cohomology is only upper semi-continuous on compact {\it complex} manifolds.
But on a compact 
{\it K\"ahler} manifold $M$ with continuously 
varying integral complex structure $J_t$, we have the Hodge decomposition
$$ \oplus_{p+q=r} H^{p,q}(M, J_t) = H^r(M, \bfC).$$
Since the dimension of $H^r(M, \bfC)$ is a topological invariant and independent of $t$
then the upper semi-continuity of the dimension of each component of the left hand side
implies that the dimension of $H^{p,q}(M, J_t)$ is independent of $t$. 

Let $\Omega_t$ be a smooth family of K\"ahler classes of $M_t$, i.e. $\Omega_t$ gives a smooth section of the vector bundle
$\{H^2(M_t)\}_{t \in B}$. Suppose that $\TildeT$ acts holomorphically on $\mathcal M \to B$ and trivially on $B$.
Thus $\TildeT$ acts on $M_t$ holomorphically for each $t \in B$.
Taking the average over the $\TildeT$-action we have a smooth family $g_t$ of $\TildeT$-invariant K\"ahler metrics such that the associated K\"ahler forms $\omega_t$ 
represent $\Omega_t$.

We denote by $L^2_{k,\TildeT}(M)$ the space of $\TildeT$-invariant real valued $L^2_k$-functions with respect to $g = g_0$ with weight $w(\mu_{\omega})$.
We shall write the $L^2$-inner product with weight  $w(\mu_{\omega})$ by $L^2(w)$.
We put 
$$H_t(M) = H^{1,1}(M_t,g_t) \cap H^2(M,\bfR).$$
For $\phi \in L^2_{k+4,\TildeT}(M)$ and $\alpha \in H_t(M)$ we put
$$ \omega_{t,\alpha,\phi} = \omega_t + \alpha + i\partial\barpartial \phi$$
which is a $\TildeT$-invariant real closed $(1,1)$-form on $(M, J_t)$ and 
$$ [\omega_{t,\alpha,\phi}] = \Omega_t + [\alpha].$$
Shrinking $B$ if necessary $H(M) = \cup_{t\in B} H_t(M)$ forms a trivial vector bundle over $B$. 
Let $h : B \times H_0(M) \to H(M)$, $(t,\alpha) \mapsto h_t(\alpha)$, be an isomorphism of vector bundles.
Note that the Sobolev spaces $L^2_{k,\TildeT}(M)$ is independent of $g_t$ for all small $t \in B$. 
Thus we consider $L^2_{k,\TildeT}$ as possessing varying norm corresponding to $g_t$. Let
$$ L^2_{k,\TildeT}(M) = i \mathcal H_{t,\alpha,\phi} \oplus W_{k,t,\alpha,\phi}$$
be the splitting of $L^2_{k,\TildeT}(M)$ into the space $i \mathcal H_{t,\alpha,\phi}$ of real Killing potentials and
its orthogonal complement $W_{k,t,\alpha,\phi}$ with respect to $\omega_{t,\alpha,\phi}$.
Let $P : L^2_{k,\TildeT}(M) \to W_{k,0}$ be the projection with respect to $g_0$. Thus $P=\pi^W_{g_0}$ in the previous notation.
Consider the map $\Phi : U \to B\times W_{k,0}$ defined on a small open neighborhood $U$ of $(0,0,0)$ in 
$B \times H_0(M) \times W_{k+4,0}$ to $B \times W_{k,0}$ by
$$\Phi(t,\alpha,\phi) = (t, P(S^{\mathrm{red}}_{v,w}(g_{t,h_t(\alpha),\phi}))).
$$
Here $g_{t,\beta,\phi}$ is the K\"ahler metric corresponding to the K\"ahler form $\omega_{t,\beta,\phi}$, and $S^{\mathrm{red}}_{v,w}(g_{t,\beta,\phi})$
is the reduced $(v,w)$-scalar curvature, c.f. Definition \ref{reducedS}.

\begin{proposition}[c.f. \cite{RST}]\label{RST1}
Let $g$ be a $(v,w)$-extremal metric where $v$ and $w$ are defined on the image of the moment map
$\mu_\omega : M \to \mathfrak t^\ast$. Let $\TildeT$ be a maximal torus in $\Isom(M,g)$ containing $T$. 
Suppose that $\TildeT$ acts holomorphically on $\mathcal M \to B$ and trivially on $B$. Then, by shrinking $B$
to a sufficiently small neighborhood of the origin if necessary, for arbitrary small perturbations 
$\Omega_t$ of the K\"ahler class $\Omega=\Omega_0$, there are weighted extremal metrics $g_t$ in $\Omega_t$.
\end{proposition}
\begin{proof}We consider the map $\Phi$ above with $g_{0,0,0}$ a $(v,w)$-extremal metric, and thus $S^{\mathrm{red}}_{v,w}(g_{0,0,0}) = 0$.
Using \eqref{Lich2.1} one can show
$$
d\Phi_{(0,0,0)} (1,\dot\alpha,\dot\phi) = \begin{pmatrix} 1 & 0\\ \ast & -L_{v,w}\dot\phi + P(dS^{\mathrm{red}}_{v,w}(\dot\alpha)) \\  \end{pmatrix}.
$$
If $\psi \in W_{k,0}$ is in the cokernel of $d\Phi_{(0,0,0)}$ then 
$$L_{v,w}\psi = 0\ \ \text{and\ \ } ( P(dS^{\mathrm{red}}_{v,w}(\dot\alpha)),\psi)_{L^2(w)} = 0.$$
But $L_{v,w}\psi = 0$ implies  that $\psi$ is a $\TildeT$-invariant real Killing potential.
Since $\TildeT$ is a maximal torus we have $\psi \in i\mathcal H_0$ and the second condition above is automatically satisfied. 
Thus
$\psi \in i\mathcal H_0 \cap W_{k,0} = \{0\}$. By the implicit function theorem the proposition follows.
\end{proof}
Instead of the maximal torus $\TildeT$, one could use a smaller torus $T'$ such that $T\subset T' \subset \TildeT$,
and argue as in \cite{RST}. Then a non-degeneracy condition considered in \cite{lebrunsimanca94} is required as in Theorem 1 in \cite{RST}. 
In fact, if we use a smaller torus $T'$ such that $T\subset T' \subset \TildeT$, then we need to take
$\mathcal H_g$ to be the space of Killing potentials corresponding to $\mathfrak t' $, the Lie algebra of $T'$.
Then $L_{v,w}\psi = 0$ implies that $\psi$ is a Killing potential but it does not imply that it belongs to 
$\mathcal H_g$, so $\psi$ needs not be zero. Hence, in order to be able to use the implicit function theorem
we need the following condition:
``If $( P(dS^{\mathrm{red}}_{v,w}(\dot\alpha)),\psi)_{L^2(w)} = 0$ for any $\dot{\alpha} \in H_0(M)$ then $\psi = 0$.''
This is the non-degeneracy condition in \cite{RST} and \cite{lebrunsimanca94} where, in the case of \cite{lebrunsimanca94},
$T=T' = \{1\}$ and $P$ is the identity. 
See also Lemma 6 
in \cite{RST}.

\section{Weighted solitons on Fano manifolds.}

In this section we consider weighted solitons on Fano manifolds which form a subclass of weighted cscK metrics.
Let $M$ be a Fano manifold, and $\omega \in 2\pi c_1(M)$ be a K\"ahler form.

\begin{definition} Let $v$ be a positive smooth function on the image of the moment map of a Hamiltonian $T$-action.
We say that $\omega$ is a weighted $v$-soliton (or simply weighted soliton, also $v$-soliton) if
$$ \Ric(\omega) - \omega = i\partial\barpartial \log v(\mu_\omega).$$
\end{definition}

Examples of $v$-solitons are\\
(i) K\"ahler-Ricci soliton for $v(\mu) = \exp(\langle\mu, \xi\rangle)$ for $\xi \in \mathfrak t$ where
the linkage with $S_{v,w}$-cscK metrics was first found by Inoue \cite{Inoue19}, \cite{Inoue22},\\
(ii) Mabuchi soliton for $v(\mu)= \langle\mu,\xi\rangle + a$ a positive affine-linear function \cite{Mab01}, and\\
(iii) base metric which induce Calabi-Yau cone metrics outside the zero sections of the canonical line bundles
(Sasaki-Einstein metrics on the associated $U(1)$-bundles) for $v(\mu) = (\ell(\xi))^{-(m+2)}$ where $\ell(\xi) = \langle\mu,\xi\rangle + a$
is a positive affine-linear function (see Proposition 2 in \cite{AJL21}).

A $T$-invariant K\"ahler form $\omega \in 2\pi c_1(M)$ is a $v$-soliton if and only if
$\omega$ is $S_{v,w} = 1$ metric with 
\begin{equation}\label{soliton0}
w(\mu) = (m+\langle d\log v,\mu\rangle)v(\mu). 
 \end{equation}
This can be seen from the formula
\begin{equation}\label{soliton1}
 S_v - w(\mu_\omega) = v(\mu_\omega)\Delta_v(\log v(\mu_\omega) - f)
 \end{equation}
where $f \in C^\infty(M)$ is the Ricci potential of $\omega$, i.e. $S - m = \Delta f$, and $\Delta_v = v^{-1}\circ \barpartial^\ast \circ v \circ \barpartial$
in which 
$v$ and $v^{-1}$ denote the multiplications by $v(\mu_\omega)$ and $v(\mu_\omega)^{-1}$.
By \eqref{soliton1} we have
$$ \int_M (S_v - w(\mu_\omega)) \omega^m = 0,$$
and thus $c_{v,w} = 1$ and 
 \begin{eqnarray*}
 \Fut_{v,w} (X) = \int_M (S_v - w(\mu_\omega))\, h_X\,\omega^m.
 \end{eqnarray*}
 Using \eqref{soliton1} this can be rewritten as 
 \begin{equation}
  \Fut_{v,w}(X) = \int_M (JX)(\log v(\mu_\omega) - f)\,v(\mu_\omega)\,\omega^m. \label{Lich6}
 \end{equation}

A characterization of the existence of weighted solitons by Ding-polystability and K-polystability was described by Li \cite{LiChi21}, Theorem 1.17 and Theorem 1.21.
The story to this result may be summarized as follows. After the resolution of Yau-Tian-Donaldson conjecture by \cite{CDS3}, \cite{Tian12}, \cite{DatarSzeke16}, \cite{CSW}
where the Gromov-Hausdorff convergence was used, a variational proof without using Gromov-Hausdorff convergence was given in \cite{BBJ} under the condition of
uniform K-stability. Further in \cite{LiChi22}, the existence was shown under the condition of $G$-uniform stability. The work of \cite{LXZ22} shows that when $G$ contains the maximal
torus $G$-uniform stability is equivalent to K-polystability. Generalizing the result of \cite{LiChi22} for K\"ahler-Einstein metrics, Han-Li \cite{HanLi20} proved the existence of weighted solitons
under the condition of $G$-uniform stability for weighted case. In \cite{BLXZ} and \cite{LiChi21}, the equivalence of $G$-uniform stability when $G$ contains the maximal
torus and K-polystability for weighted case was shown.

\section{Geometry of Kuranishi family}

Let $\varpi : \mathfrak M \to B$ be the Kuranishi family of a Fano manifold $M$ satisfying \eqref{Kura}
as described in Section 1. Then the $v$-soliton $\omega$ on $M_0 = M$ remains to be K\"ahler forms
on $M_t,\ t\in B$,
by Theorem 1.4 in \cite{FSZ22}. Further, it was shown in Theorem 1.5 in \cite{FSZ22} that the Ricci form $\Ric(M_t,\omega)$
of $(M_t, \omega)$ is given by
\begin{eqnarray}\label{Kura2}
 \Ric(M_t,\omega) 
 = \omega + \partial_t\barpartial_t (f_0 + \log\det(I - \varphi(t)\overline{\varphi(t)})).
 \end{eqnarray}
 But since we assume $(M_0, \omega_0)$ with $\omega_0 = \omega$ is a $v$-soliton we have
 \begin{eqnarray}\label{Kura2.1}
 f_0 = \log v(\mu_{\omega_0}).
 \end{eqnarray}
Recall that $\varphi(t)$ in \eqref{Kura} can be considered as
\begin{eqnarray}
{}&&\varphi(t) \in A^{0,1}(T'M_0) \cong \Gamma(\operatorname{Hom}(T^{\prime\ast}M_0,T^{\prime\prime\ast}M_0)) = \Gamma(T^{\prime}M_0 \otimes T^{\prime\prime\ast}M_0), \nonumber\\
&& \quad \varphi(t) = \varphi^i{}_\barj (t) \frac{\partial}{\partial z^i} \otimes dz^\barj. \label{D1}
\end{eqnarray}
Here $z^i = z_0^i$ are local holomorphic coordinates of $M_0$, and we keep this notation below.
Then, $T^{\prime\ast}M_t$ is spanned by
\begin{eqnarray}\label{D2}
e^i := dz^i + \varphi^i{}_\barj(t) dz^\barj , \qquad i = 1, \cdots, m,
\end{eqnarray}
or equivalently, $T^{\prime\prime}M_t$ is spanned by 
\begin{eqnarray}\label{D2.1}
T_\barj := \frac{\partial}{\partial z^\barj} - \varphi^i{}_\barj(t) \frac{\partial}{\partial z^i}, \qquad j= 1, \cdots, m.
\end{eqnarray}

\begin{lemma}\label{Kura3}
Suppose that (2) of Theorem \ref{Main Thm} is satisfied. Then the identity component $\Aut^T_0(M)$ of $\Aut^T(M)$ 
acts on $H^1(M_0, T'M_0)\cong T'_0B$ trivially, and hence on $B$ trivially.
\end{lemma}
\begin{proof}
We closely follow the arguments in \cite{CaoSunYauZhang2022}, page 823. But some missing
computations in \cite{CaoSunYauZhang2022} are supplemented, which are \eqref{p5} - \eqref{q2} below, for the reader's convenience. 
Since the Kuranishi family is a complex analytic family (Proposition 2.6, 2.7 in Chapter 4 of \cite{MorrowKodaira}, or Theorem 6.5, \cite{Kodaira86}), 
by the assumption there are $T$-invariant holomorphic vector fields $v_1(t),\, \cdots, v_\ell(t)$ which form a basis of $H^0(M_t,T'M_t)^T$ and holomorphic in t.
We regard these are vector fields on $M_0 \cong M$ since all $M_t$ are diffeomorphic to $M_0$. Since 
$(I -\varphi\barvarphi)^{-1} - \barvarphi(I-\varphi\barvarphi)^{-1}$ is invertible for small $t$ we may put
$$ \wtv_p := ((I -\varphi\barvarphi)^{-1} - \barvarphi(I-\varphi\barvarphi)^{-1})^{-1}v_p.$$
Let $z^1,\, \cdots,\, z^m$ and $w^1, \, \cdots,\, w^m$ be local holomorphic coordinates for $M_0$ and $M_t$ respectively defined on a common open set $U$ of $M$.
Note that (4.4) and (4.5) in \cite{FSZ22} imply
\begin{equation}\label{p5}
((I - \varphi\barvarphi)^{-1})^i{}_j = \frac{\partial z^i}{\partial w^\alpha} \frac{\partial w^\alpha}{\partial z^j}
\end{equation}
and
\begin{equation}\label{p5.1}
 - \overline{\varphi^i{}_\barj}\, ((I - \varphi\barvarphi)^{-1})^j{}_\ell = \frac{\partial z^\bari}{\partial w^\alpha} \, \frac{\partial w^\alpha}{\partial z^j} .
\end{equation}
Then we can see using \eqref{p5} and \eqref{p5.1} that
\begin{equation}\label{p4}
v_p= \wtv_p^j\,\frac{\partial w^\alpha}{\partial z^j}\,\frac{\partial}{\partial w^\alpha}.
\end{equation}
Since $v$ is holomorphic on $M_t$, we have $T_\barj v^\alpha = 0$, that is,
\begin{equation}\label{q1}
\lb\frac{\partial}{\partial z^\barj} - \varphi^i{}_\barj(t) \frac{\partial}{\partial z^i}\rb\lb\frac{\partial w^\alpha}{\partial z^k}\,\wtv_p^k\rb = 0.
\end{equation}
On the other hand
\begin{equation}\label{q2}
\lb T_\barj\lb\frac{\partial w^\alpha}{\partial z^k}\rb\rb\,\wtv_p^k = (\wtv_p\varphi)^i{}_\barj \, \frac{\partial w^\alpha}{\partial z^i}.
\end{equation}
From \eqref{q1} and \eqref{q2} we get
\begin{equation}\label{q3}
\barpartial_0\,\wtv_p = - [\wtv_p,\varphi].
\end{equation}
Since $\varphi(0) = 0$ we obtain
\begin{equation}\label{q4}
\barpartial_0\,\lb \left.\frac{\partial}{\partial t_k}\right|_{t=0} \wtv_p(t)\rb = - [\wtv_p,\varphi_k].
\end{equation}
This implies the infinitesimal generators of $\Aut_0^T(M)$ acts on $H^1(M_0, T'M_0)$ trivially. This completes the proof.
\end{proof}

\begin{proof}[Proof of Theorem \ref{Main Thm}]
We first prove that (2) implies (1). 
Let $G:=\Isom_0^T(M_0,\omega)$ be the identity component of the $T$-equivariant isometries of $(M_0,\omega)$ so that $G$ preserves both $\omega$
and $J_0$. 
Then since $\omega = \omega_0$ is a weighted $v$-soliton it is a $(v,w)$-cscK metric with $w(\mu) = (m + \langle d \log v, \mu\rangle) v(\mu)$
and $G^\bfC = \Aut_0^T(M)$. 
By Lemma \ref{Kura3}, $G$ acts on $H^1(M_0, T'M_0)$ trivially, which implies that $G$ preserves $\varphi(t)$ 
since $\varphi(t)$ is uniquely determined by $\sum_{i=1}^k t^i\varphi_i$ in Kuranishi's equation \eqref{Kura}. Hence $G$ also preserves $J_t$,
and thus $G \subset \Isom_0^T(M_t, \omega)$. But since 
$$\dim G^\bfC = \dim \Aut^T(M) = \dim \Aut^T(M_t) \ge \dim_\bfR \Isom_0^T(M_t,\omega)$$
we have $G = \Isom_0^T(M_t, \omega)$.
This implies that the Hamiltonian vector fields for $(M_0, \omega)$ remain to be Hamiltonian vector fields of $(M_t,\omega)$, and
the moment map $\mu_{\omega_t}$ is unchanged as $t$ varies. Thus 
\begin{equation} \label{Lich6.1}
v(\mu_{\omega_t}) = v(\mu_{\omega})
\end{equation}
for all $t \in B$.

Let $\TildeT$ be the maximal torus in $G$ containing $T$. Then since $\TildeT \subset \Aut_0^T(M)$ Lemma \ref{Kura3} 
implies that $\TildeT$ acts on $B$ trivially. By Proposition \ref{RST1}, shrinking $B$ if necessary, $M_t$ admits a $(v,w)$-extremal metric
for any $t \in B$. We wish to show this $(v,w)$-extremal metric is a $(v,w)$-cscK metric so that it is a $v$-soliton. To see this, by Remark \ref{Lich5.1},  it is sufficient to show 
the invariant $\Fut_{v,w}(t)$ in \eqref{Lich6} for $M_t$ vanishes. By \eqref{Kura2} and \eqref{Kura2.1}, we need to take $f$ in \eqref{Lich6},  to be
$$
f_t := \log v(\mu_{\omega_0}) + \log \det(I - \varphi(t)\overline{\varphi(t)}).
$$
Hence using \eqref{Lich6.1} we have
 \begin{equation} \label{Lich7}
  \Fut_{v,w}(t)(X) = - \int_M (JX)(\log \det(I - \varphi(t)\overline{\varphi(t)}))\,v(\mu_\omega)\,\omega^m.
 \end{equation}
But since any automorphism of $M_t$ preserves $\varphi(t)$ the derivative by $JX$ on the right hand side of \eqref{Lich7} vanishes.
Thus $\Fut_{v,w}(t)$ vanishes, 
and by Remark \ref{Lich5.1} the extremal $(v,w)$-extremal metric must be a $v$-soliton. This proves that (2) implies (1).

Next we prove that (1) implies (3). We first show the action of $G := \Isom_0^T(M_0,\omega)$ and $G^\bfC$ on $B$ is trivial. 

For this purpose we show that if this is not the case then there is a  non-product $T^\bfC$-equivariant test configuration 
 $\{(M_t, K_{M_t}^{-k})\}$ using 
arguments similar to \cite{CaoSunYauZhang2022}, page 822-823. 
Because of the construction of the Kuranishi family, the nontrivial action of $G^\bfC$ on $B$ induces
a one parameter subgroup $\lambda : \bfC^\ast \to G^\bfC$ whose $T^\bfC$-equivariant action on
$T'_0B \cong H^1(M_0, T'M_0)$ is nontrivial. 
We can choose a basis $e_1, \cdots, e_\ell$ of $H^1(M_0, T'M_0)$ 
such that $\lambda(s) e_i = s^{\kappa_i} e_i$ with $\kappa_i \in \bfZ$. Since this action is nontrivial some 
$\kappa_i$ is non-zero, and we choose and fix one of such $i$'s, and we may assume $\kappa_i > 0$ by replacing $\lambda$ 
by $\lambda^{-1}$. Consider the one-dimensional subfamily $\{ M_t \ |\ t=(0, \cdots, 0, t_i, 0, \cdots, 0),
| t_i | < \epsilon  \}$ of $\mathcal M \to B$ for small $\epsilon > 0$. Then we have an action of $\{s \ |\ |s| < 1\}$ 
corresponding to the $\lambda$-action expressed by $M_t \to M_{st}$.
All $M_t$ with $t \ne 0$ are biholomorphic 
because of the action of $\{s \ |\ 0 < |s| < 1\}$. Then the Kodaira-Spencer map $T'B_t \to H^1(M_t, T'M_t)$ is only surjective (see e.g. Theorem 2.1, (3) in
\cite{CaoSunYauZhang2022}) but
not isomorphic for $t\ne 0$, while $T'B_0 \to H^1(M_0, T'M_0)$ is isomorphic. It follows that, 
for $t \ne 0$, $M_{t}$ is not biholomorphic to $M_0$. Hence after a suitable base change we obtain a
non-product $T^\bfC$-equivariant test configuration  $\{(M_t, K_{M_t}^{-k})\}$.

But this is impossible since $M_t$ has a $v$-soliton and K-polystable with respect to $T^\bfC$-equivariant test
configurations, the central fiber $M_0$ also has a
 $v$-soliton and the Donaldson-Futaki invariant is zero (see Theorem 1.17 and 1.21 in \cite{LiChi21}, or
 \cite{AJL21}, or Theorem 1.0.7 in \cite{Hallam22}, or \cite{Hallam23}). Thus the action of $G$ on $B$ is trivial. 
 
Then as we argued at the beginning of this proof, $G$ preserves both $\omega$ and $\varphi(t)$, and thus we have an inclusion
$G \subset \Isom_0(M_t,\omega)$. In particular $T \subset \Isom_0(M_t,\omega) \subset \Aut_0(M_t)$ and 
$G \subset \Isom_0^T(M_t,\omega)$, 
$G^\bfC = \Aut_0^T(M_0) \subset \Aut_0^T(M_t)$. But since $\dim H^0(M_t, T^\prime M_t)^T$ is upper semi-continuous we obtain
$G^\bfC = \Aut^T(M_t)$ for all $t \in B$. This proves that (1) implies (3). That (3) implies (2) is trivial. This completes the proof
of Theorem \ref{Main Thm}
\end{proof}

\end{document}